\documentclass[amsthm]{elsart}

\usepackage{amsmath}
\usepackage{amssymb}
\usepackage{graphicx}

\newcommand{\D}{\mathbb D}
\newcommand{\DD}{\mathcal D}
\newcommand{\HH}{\mathcal H}
\newcommand{\JJ}{\mathcal J}
\newcommand{\Lcl}{\mathbb L}
\newcommand{\LL}{\mathcal L}
\newcommand{\mat}{{\rm Mat}}
\newcommand{\Rcl}{\mathbb R}
\newcommand{\RR}{\mathcal R}
\newcommand{\spn}{\text{span}}
\newcommand{\supp}{\text{supp}\,}
\newcommand{\Z}{\mathbb Z}

\newcounter{tmp}
\newcounter{assum}
\newcommand{\assump}[1]{
\setcounter{tmp}{\value{thm}}
\setcounter{thm}{\value{assum}}
\begin{assum}#1\end{assum}
\setcounter{thm}{\value{tmp}}
\stepcounter{assum}
}

\begin{document}

\begin{frontmatter}
\title{Cellularity of Diagram Algebras as Twisted Semigroup Algebras}
\author{Stewart Wilcox\thanksref{label1}}
\thanks[label1]{This work was completed while the author was a postgraduate in the Department of Mathematics 
and Statistics at the University of Sydney}
\ead{stewartw@math.harvard.edu}
\address{Harvard University,\\
               Department of Mathematics,\\
	       1 Oxford Street,
	       Cambridge, MA 02138,\\
	       U.S.A.}

\begin{abstract}
The Temperley-Lieb and Brauer algebras and their cyclotomic analogues, as well as the partition algebra, are all examples of twisted 
semigroup algebras. We prove a general theorem about the cellularity of twisted semigroup algebras of regular semigroups. This 
theorem, which generalises a recent result of East about semigroup algebras of inverse semigroups, allows us to easily reproduce 
the cellularity of these algebras.
\end{abstract}

\begin{keyword}
cellular \sep twisted semigroup algebra
\end{keyword}
\end{frontmatter}
\renewcommand{\theenumi}{\roman{enumi}}
\section{Introduction}
There has been much interest in algebras which have a basis consisting of diagrams, which are multiplied in some natural 
diagrammatic way. Examples of these so-called \emph{diagram algebras} include the partition, Brauer and Temperley-Lieb algebras. 
These three examples have been studied extensively in the literature. In particular each has been shown to be \emph{cellular}; this 
property, introduced by Graham and Lehrer in \cite{gus}, allows us to easily derive information about the semisimplicity of the 
algebra and about its representation theory, even in the non-semisimple case.

In the three algebras mentioned above, the product of two diagram basis elements is always a scalar multiple of another basis 
element. Motivated by this observation, we realise these algebras as \emph{twisted semigroup algebras}. We can then reproduce the 
above cellularity results by proving a general theorem about twisted semigroup algebras, which extends a recent result of East 
\cite{james}.
\section{Semigroups}
Central to the study of any semigroup are certain relations defined by Green \cite{green}, which we now briefly recall. Let $S$ 
be a semigroup. Write $x\leq_\RR y$, $x\leq_\LL y$ or $x\leq_\JJ y$ if $x$ can be obtained from $y$ by, respectively, left 
multiplication, right multiplication or simultaneous left and right multiplication. Green's relations are the equivalence relations 
defined by
\begin{eqnarray*}
  &\RR=\leq_\RR\cap\geq_\RR\hspace{10mm}
  \LL=\leq_\LL\cap\geq_\LL\hspace{10mm}
  \JJ=\leq_\JJ\cap\geq_\JJ\hspace{10mm}&\\
  &\HH=\RR\cap\LL\hspace{10mm}
  \DD=\langle\RR\cup\LL\rangle&
\end{eqnarray*}
where the final expression denotes the equivalence relation generated by $\RR$ and $\LL$. Let $\D$ denote the set of equivalence 
classes of $\DD$ in $S$, or \emph{$\DD$ classes}. For $D\in\D$, let $\Lcl_D$ and $\Rcl_D$ denote the sets of $\LL$ and $\RR$ 
classes in $D$ respectively. The following property of Green's relations, along with its dual, constitutes a fundamental result 
known as Green's Lemma.
\begin{lem}[Green's Lemma \cite{green}]
Suppose that $x\in S$ and $a\in S$ are such that $xa\;\RR\;x$. Then right multiplication by $a$ gives an $\RR$ class preserving 
bijection from the $\LL$ class of $x$ to the $\LL$ class of $xa$.
\end{lem}
A semigroup $S$ is said to be \emph{group bound} if for each $x\in S$, there exists a positive integer $n$ such that $x^n$ lies 
in a subgroup of $S$. In particular, every finite semigroup is group bound. The following results are well known.
\begin{thm}\label{known}
Suppose $S$ is a group bound semigroup. Then
\begin{enumerate}
\item The relations $\JJ$ and $\DD$ coincide.
\item If $x\;\DD\;xy$ then $x\;\RR\;xy$.
\item If $y\;\DD\;xy$ then $y\;\LL\;xy$.
\end{enumerate}
\end{thm}
Recall also that a semigroup $S$ is \emph{regular} if, for each $x\in S$, there exists $y\in S$ such that $xyx=x$. Equivalently 
$S$ is regular if each $\DD$ class contains an idempotent.
\section{Twisted Semigroup Algebras}
By analogy with twisted group algebras \cite{passman}, we define a twisted semigroup algebra. The following definition is 
essentially that in \cite{twisted}, except that we give no special treatment to the zero of the semigroup (if it exists).
\begin{defn}
Suppose $S$ is a semigroup and $R$ is a commutative ring with $1$. A \emph{twisting} from $S$ into $R$ is a map
\[
  \alpha:S\times S\rightarrow R
\]
which satisfies
\begin{equation}\label{twistassoc}
  \alpha(x,y)\alpha(xy,z)=\alpha(x,yz)\alpha(y,z)
\end{equation}
for all $x,y,z\in S$. The \emph{twisted semigroup algebra} of $S$ over $R$, with twisting $\alpha$, denoted by 
$R^\alpha[S]$, is the $R$-algebra with $R$-basis $S$ and multiplication $\cdot$ defined by
\[
  x\cdot y=\alpha(x,y)(xy)
\]
for $x,y\in S$, and extended by linearity. It follows easily from (\ref{twistassoc}) that $R^\alpha[S]$ is associative.
\end{defn}
For $T\subseteq S$, let $R^\alpha[T]$ denote the $R$-span of $T$ in $R^\alpha[S]$, so that $T$ forms an $R$-basis for 
$R^\alpha[T]$. It is clear that if $T$ is a subsemigroup of $S$, then $R^\alpha[T]$ is a subalgebra, and moreover is 
isomorphic to the twisted semigroup algebra of $T$ whose twisting is the restriction of $\alpha$ to $T$, thus justifying the 
notation.
\section{Cellular Algebras}
Cellular algebras were introduced in the famous paper of Graham and Lehrer \cite{gus}. Although the definition in \cite{gus} requires the 
algebra to be unital, it is easy to see that this does not affect the theory significantly.
\begin{defn}\label{gus}
Suppose that $R$ is a commutative ring with identity. Recall that an anti-involution $*$ on 
an $R$-algebra $A$ is an $R$-linear map from $A$ to $A$ such that
\[
  (a^*)^*=a
  \hspace{10mm}\text{and}\hspace{10mm}
  (ab)^*=b^*a^*
\]
for $a$ and $b\in A$. An associative $R$-algebra $A$ is \emph{cellular}, with \emph{cell datum} $(\Lambda,M,C,*)$, if
\renewcommand{\theenumi}{C\arabic{enumi}}
\begin{enumerate}
\item
  $\Lambda$ is a finite poset, and for each $\lambda\in\Lambda$ we have a finite indexing set $M(\lambda)$ and elements 
  $C^\lambda_{st}\in A$ for $s,t\in M(\lambda)$. The elements
  \[
    \{C^\lambda_{st}\mid\lambda\in\Lambda\text{ and }s,t\in M(\lambda)\}
  \]
  form an $R$-basis of $A$.
\item
  The map $*:A\rightarrow A$ is an anti-involution, whose action on the above basis is given by
  \[
    \left(C^\lambda_{st}\right)^*=C^\lambda_{ts}.
  \]
\item
  For any $\lambda\in\Lambda$, $s\in M(\lambda)$ and $a\in A$, there exist elements $r_a(s',s)\in R$ for 
  $s'\in M(\lambda)$ such that, for each $t\in M(\lambda)$,
  \[
    aC^\lambda_{st}\in\sum_{s'\in M(\lambda)}r_a(s',s)C^\lambda_{s't}
      +A(<\lambda)
  \]
  where
  \[
    A(<\lambda)={\rm span}_R\{C^\mu_{s''t''}\mid\mu<\lambda\text{ and }s'',t''\in M(\mu)\}.
  \]
\end{enumerate}
\renewcommand{\theenumi}{\roman{enumi}}
\end{defn}
\section{The Main Theorem}
In this section we prove a version of Theorem 15 of \cite{james} for regular semigroups and for twisted semigroup algebras. 
As in \cite{james}, we will assume that the group algebras of the maximal subgroups of $S$ are cellular, namely in Assumption 
\ref{assumgrp}. However, in \cite{james} the anti-involutions on these algebras are woven together in the hope of creating an 
anti-involution on the semigroup algebra. In contrast, we start at the top by constructing an anti-involution $*$ on the 
semigroup algebra, and assuming that the anti-involution on each group algebra is a restriction of $*$.	Therefore the assumptions 
we make will ensure that there is an anti-involution on the semigroup which induces an anti-involution on the twisted semigroup 
algebra, and which fixes certain maximal subgroups setwise.

We find it convenient to list the assumptions in the following discussion before stating the theorem. Firstly, we begin with the 
following objects.
\assump{\label{assum1}
Let $S$ be a finite semigroup, $*:S\rightarrow S$ an anti-involution, $R$ a commutative ring with identity, and $\alpha$ a 
twisting from $S$ into $R$.
}
We suppose that $*$ and $\alpha$ are compatible in the following sense.
\assump{\label{twiststar}
Assume that
\[
  \alpha(x,y)=\alpha(y^*,x^*)
\]
for all $x,y\in S$.
}
This assumption implies that $*$ extends to an $R$-linear anti-involution on $R^\alpha[S]$, which we also denote by $*$. The next assumption 
ensures that $*$ fixes certain maximal subgroups, and also that $S$ is regular.
\assump{\label{1D}
Suppose that for each $\DD$ class $D\in\D$, we have an idempotent $1_D\in D$ which is fixed by $*$.
}
Let $L_D$ denote the $\LL$ class of $1_D$, so $L_D^*$ is the $\RR$ class of $1_D$. The $\HH$ class $G_D=L_D\cap L_D^*$ of $1_D$ is 
a group. Moreover $*$ fixes $G_D$, and we denote its restriction to $G_D$ by $*$. We will need a certain twisted group algebra over $G_D$ to 
be cellular. However, for this to give information about the rest of the $\DD$ class $D$, we need the scalar elements $\alpha(x,y)$ to 
be ``sufficiently invertible". The following assumption, although very unnatural, gives us the generality we require. It essentially 
says that although $\alpha$ may not be invertible, when restricted to $L_D\times L_D^*$ it can be decomposed into a constant part and 
an invertible part. We use $G(R)$ to denote the group of units of $R$.
\assump{\label{beta}
For each $\DD$ class $D$, we assume the existence of a map
\[
  \beta:L_D\times L_D^*\rightarrow G(R)
\]
which satisfies the following analogues of (\ref{twistassoc}) and Assumption \ref{twiststar}:
\begin{eqnarray}
  \beta(x,y)\beta(xy,z)&=&\beta(x,yz)\beta(y,z),\label{beta1}\\
  \alpha(x,y)\beta(xy,z)&=&\alpha(x,yz)\beta(y,z),\text{ and}\label{beta2}\\
  \beta(x,y)&=&\beta(y^*,x^*)\label{beta3}
\end{eqnarray}
whenever the relevant values of $\beta$ are defined.
}
Before proceeding, we discuss the implications of Assumption \ref{beta}. By (\ref{beta1}), the restriction of $\beta$ defines a twisting from $G_D$ 
into $R$. Also as above, (\ref{beta3}) implies that $*$ induces an anti-involution on $R^\beta[G_D]$, which we again denote by $*$. Now replacing $x$, 
$y$ and $z$ with $z^*$, $y^*$ and $x^*$ respectively in (\ref{beta2}), and employing Assumption \ref{twiststar} and (\ref{beta3}), we obtain
\begin{equation}\label{beta4}
  \beta(x,yz)\alpha(y,z)=\beta(x,y)\alpha(xy,z)
\end{equation}
whenever the values of $\beta$ are defined. As foreshadowed, the restriction of $\alpha$ to $L_D\times L_D^*$ can be 
obtained from $\beta$ by multiplying by a constant. Indeed putting $x=y=1_D$ in (\ref{beta2}), we obtain
\[
  \alpha(1_D,1_D)\beta(1_D,z)=\alpha(1_D,z)\beta(1_D,z)
\]
for $z\in L_D^*$. Since $\beta(1_D,z)$ is invertible, this gives $\alpha(1_D,z)=\alpha(1_D,1_D)$. Similarly putting $y=z=1_D$ in (\ref{beta1}) gives 
$\beta(x,1_D)=\beta(1_D,1_D)$ for $x\in L_D$. Finally for $x\in L_D$ and $z\in L_D^*$, putting $y=1_D$ in (\ref{beta4}) gives
\[
  \beta(x,z)\alpha(1_D,z)=\beta(x,1_D)\alpha(x,z).
\]
Thus $\beta(x,z)\alpha(1_D,1_D)=\beta(1_D,1_D)\alpha(x,z)$, so that
\[
  \alpha(x,z)=\alpha(D)\beta(x,z),
\]
where $\alpha(D)=\alpha(1_D,1_D)\beta(1_D,1_D)^{-1}$. In particular, multiplication by $\alpha(D)$ gives a homomorphism 
$R^\alpha[G_D]\rightarrow R^\beta[G_D]$. 

As foreshadowed, our final assumption is that certain twisted group algebras of the maximal subgroups are cellular. 
\assump{\label{assumgrp}
Suppose that, for each $\DD$ class $D$, the twisted group algebra $R^\beta[G_D]$ is cellular with cell datum
\[
  (\Lambda_D,M_D,C,*).
\]
}
Note we have assumed that the anti-involution in this cell datum is exactly $*$. Under these assumptions, we will show that 
the twisted semigroup algebra $R^\alpha[S]$ is cellular. To be more precise, we describe the cell datum below. 
Because $S$ is finite, $\DD=\JJ$ by (i) of Theorem \ref{known}, so we have a relation $\leq_\DD$ on $S$. Define the poset
\[
  \Lambda=\{(D,\lambda)\mid D\in\D\text{ and }\lambda\in\Lambda_D\}
\]
with partial order
\[
  (D_1,\lambda_1)\leq(D_2,\lambda_2)\text{ iff }D_1<_\DD D_2
    \text{ or }D_1=D_2\text{ and }\lambda_1\leq\lambda_2\text{ in }\Lambda_{D_1}.
\]
Now for $(D,\lambda)\in\Lambda$, let
\[
  M(D,\lambda)=\Lcl_D\times M_D(\lambda).
\]
Finally for each $L\in\Lcl_D$, choose any $u_L\in L$ with $u_L\;\RR\;1_D$. The basis elements that result from the cell datum of 
$R^\beta[G_D]$ can be written uniquely as
\[
  C^\lambda_{st}=\sum_{g\in G_D}c^\lambda_{st}(g)g
\]
for some coefficients $c^\lambda_{st}(g)\in R$. Define
\[
  C^{(D,\lambda)}_{(L,s)(K,t)}=\sum_{g\in G_D}c^\lambda_{st}(g)\beta(u_L^*,g)\beta(u_L^*g,u_K)(u_L^*gu_K)
    \in R^\alpha[S]
\]
for each $(D,\lambda)\in\Lambda_D$ and $(L,s),(K,t)\in M(D,\lambda)$.
\begin{thm}\label{main}
Under Assumptions \ref{assum1}, \ref{twiststar}, \ref{1D}, \ref{beta} and \ref{assumgrp}, the algebra $R^\alpha[S]$ is 
cellular with the cell datum
\[
  (\Lambda,M,C,*)
\]
as given above.
\end{thm}
As mentioned, Assumption \ref{beta} is very unnatural. However, we are primarily interested in two special cases. The first 
is the most natural, and applies when the twisting elements $\alpha(x,y)$ are invertible. In particular this includes the case 
of a semigroup algebra, in which the twisting is trivial.
\begin{cor}\label{mainunits}
Suppose Assumptions \ref{assum1}, \ref{twiststar} and \ref{1D} hold. Suppose also that for each $D\in\D$ and for each 
$x\;\LL\;1_D$ and $y\;\RR\;1_D$, the element $\alpha(x,y)\in R$ is invertible. As in Assumption \ref{assumgrp}, suppose that 
$R^\alpha[G_D]$ is cellular with cell datum
\[
  (\Lambda_D,M_D,C,*).
\]
Then the algebra $R^\alpha[S]$ is cellular with the cell datum
\[
  (\Lambda,M,C,*),
\]
where $\Lambda$, $M$ and $*$ are as given above. The basis elements now take the more elegant form
\[
  C^{(D,\lambda)}_{(L,s)(K,t)}=u_L^*\cdot C^\lambda_{st}\cdot u_K.
\]
\end{cor}
This follows from Theorem \ref{main} by setting $\beta$ to be the relevant restriction of $\alpha$ for each $D$ class. The 
second special case will aid our investigation of the Brauer, Temperley-Lieb and partition algebras.
\begin{cor}\label{mainconst}
Suppose Assumptions \ref{assum1}, \ref{twiststar} and \ref{1D} hold. Suppose also that we have $\alpha(x,y)=\alpha(x,z)$ 
whenever $y\;\RR\;z$. Suppose that the group algebra $R[G_D]$ is cellular with cell datum
\[
  (\Lambda_D,M_D,C,*).
\]
Then the algebra $R^\alpha[S]$ is cellular with the cell datum
\[
  (\Lambda,M,C,*),
\]
where $\Lambda$, $M$ and $*$ are as given above. The basis elements now take the form
\[
  C^{(D,\lambda)}_{(L,s)(K,t)}=\sum_{g\in G_D}c^\lambda_{st}(g)(u_L^*gu_K).
\]
\end{cor}
This follows from Theorem \ref{main} be setting $\beta(x,y)=1$. To verify (\ref{beta2}) of Assumption \ref{beta} in this case, suppose $\beta(xy,z)$ and 
$\beta(y,z)$ are defined, so that $y\in L_D$ and $z\in L_D^*$ for some $D\in\D$. Then $1_Dz=z$, so Green's Lemma shows that right multiplication by $z$ is an 
$\RR$ class preserving map on $L_D$. In particular $yz\;\RR\;y$, so $\alpha(x,y)=\alpha(x,yz)$ as required.

The proof of Theorem \ref{main} contains many notationally unpleasant calculations related to associativity. To partially alleviate 
this, we introduce a partial product on $R^\alpha[S]$. For each $D\in\D$, define
\[
  \circ:R^\alpha[L_D]\times R^\alpha[L_D^*]\rightarrow R^\alpha[S]
\]
by setting $x\circ y=\beta(x,y)(xy)$ for $x\in L_D$ and $y\in L_D^*$, and extending by $R$-linearity. It will often be necessary to check that the arguments 
of $\circ$ lie in $R^\alpha[L_D]$ and $R^\alpha[L_D^*]$ respectively, for the appropriate $D$; we generally leave this to the reader. It should be noted that 
in the special case of Corollary \ref{mainunits}, this product coincides with $\cdot$, so the associativity of $\cdot$ makes many of the tedious calculations trivial; 
thus a direct proof of this case is much more natural, and still contains the essential ideas.

Note that $R^\beta[G_D]$ is equal to $R^\alpha[G_D]$ as an $R$-module, and the product on $R^\beta[G_D]$ is just the restriction of $\circ$. Also the above 
definition of $C^{(D,\lambda)}_{(L,s)(K,t)}$ now becomes
\[
  C^{(D,\lambda)}_{(L,s)(K,t)}=(u_L^*\circ C^\lambda_{st})\circ u_K
\]
for $(D,\lambda)\in\Lambda_D$ and $(L,s),(K,t)\in M(D,\lambda)$. Applying linearity to equations (\ref{beta1}), (\ref{beta2}), (\ref{beta3}) and 
(\ref{beta4}) respectively, we obtain:
\begin{eqnarray}
  (a\circ b)\circ c&=&a\circ(b\circ c),\label{beta1'}\\
  (a\cdot b)\circ c&=&a\cdot(b\circ c)\text{ if }(\supp a)(\supp b)\subseteq L_D,\label{beta2'}\\
%
%
%
%
%
  (a\circ b)^*&=&b^*\circ a^*,\text{ and}\label{beta3'}\\
  (a\circ b)\cdot c&=&a\circ(b\cdot c)\text{ if }(\supp b)(\supp c)\subseteq L_D^*\label{beta4'}
\end{eqnarray}
for any $a$, $b$ and $c\in R^\alpha[S]$, whenever the relevant values of $\circ$ are defined. Here $\supp a$ is the set of elements of $S$ which appear with 
nonzero coefficient in $a$. We now give a proof Theorem \ref{main}, which for clarity we separate into three lemmas corresponding to properties (C1), (C2) 
and (C3) of Definition \ref{gus}.  
\begin{lem}\label{C1lem}
The elements
\[
  \left\{\left.C^{(D,\lambda)}_{(L,s)(K,t)}\;\right|\;(D,\lambda)\in\Lambda
    \text{ and }(L,s),(K,t)\in M(D,\lambda)\right\}
\]
form an $R$-basis for $R^\alpha[S]$.
\end{lem}
\begin{proof}
Consider a $\DD$ class $D\in\D$. Now $*$ preserves $\DD$ and $1_D^*=1_D$, so $*$ maps $D$ onto $D$. Since $*$ is an 
anti-involution, it therefore maps the $\LL$ classes in $D$ bijectively onto the $\RR$ classes in $D$. That is, 
each $\RR$ class in $D$ is uniquely expressible as $L^*$ for some $L\in\Lcl_D$. Thus each $\HH$-class in $D$ is 
uniquely expressible as $L^*\cap K$ for some $L,K\in\Lcl_D$.

For each $L\in\Lcl_D$, we have $u_L\;\RR\;1_D$ by choice of $u_L$. Since $1_D$ is idempotent, this implies that $1_Du_L=u_L$. 
By Green's Lemma, right multiplication by $u_L$ then gives an $\RR$ class preserving bijection from the $\LL$ class of $1_D$ to 
the $\LL$ class of $u_L$, which is $L$. Applying $*$ we have $u_L^*1_D=u_L^*$, so left multiplication by $u_L^*$ gives an $\LL$ 
class preserving bijection from the $\RR$ class of $1_D$ to the $\RR$ class of $u_L^*$, namely $L^*$. We therefore have two bijections
\[
  G_D\rightarrow L^*\cap L_D\rightarrow L^*\cap K
\]
given respectively by $g\mapsto u_L^*g$ and $x\mapsto xu_K$. Thus we have $R$-module homomorphisms
\[
  R^\beta[G_D]=R^\alpha[G_D]\rightarrow R^\alpha[L^*\cap L_D]\rightarrow R^\alpha[L^*\cap K]
\]
given respectively by $a\mapsto u_L^*\circ a$ and $a\mapsto a\circ u_K$. On the natural bases these homomorphisms are 
given by
\begin{eqnarray*}
  g&\mapsto&\beta(u_L^*,g)(u_L^*g)\hspace{10mm}\text{for }g\in G_D,\text{ and}\\
  x&\mapsto&\beta(x,u_K)(xu_K)\hspace{10mm}\text{for }x\in L^*\cap L_D.\\
\end{eqnarray*}
Because the elements $\beta(x,y)$ are invertible, and the above maps between the natural bases are bijections, these homomorphisms are 
$R$-module isomorphisms. Now the elements
\[
  \left\{\left.C^\lambda_{st}\;\right|\;\lambda\in\Lambda_D\text{ and }s,t\in M_D(\lambda)\right\}
\]
form an $R$-basis for $R^\beta[G_D]$, so applying the above isomorphisms, the elements
\begin{eqnarray*}
  \left\{\left.\left(u_L^*\circ C^\lambda_{st}\right)\circ u_K\;\right|\;
      \lambda\in\Lambda_D\text{ and }s,t\in M_D(\lambda)\right\}\hspace{-40mm}&&\\
    &=&\left\{\left.C^{(D,\lambda)}_{(L,s)(K,t)}\;\right|\;
      \lambda\in\Lambda_D\text{ and }s,t\in M_D(\lambda)\right\}
\end{eqnarray*}
form an $R$-basis for $R^\alpha[L^*\cap K]$. Now $D$ is a disjoint union of its $\HH$ classes
\[
  D=\coprod_{L,K\in\Lcl_D}L^*\cap K,
\]
and $S$ is in turn a disjoint union of its $\DD$ classes
\[
  S=\coprod_{D\in\D}D.
\]
Thus
\[
  R^\alpha[D]=\bigoplus_{L,K\in\Lcl_D}R^\alpha[L^*\cap K]
\]
and
\[
  R^\alpha[S]=\bigoplus_{D\in\D}R^\alpha[D],
\]
so that
\begin{equation}\label{Dbasis}
  \left\{\left.C^{(D,\lambda)}_{(L,s)(K,t)}\;\right|\;\lambda\in\Lambda_D
    \text{ and }(L,s),(K,t)\in M(D,\lambda)\right\}
\end{equation}
form an $R$-basis for $R^\alpha[D]$, and
\[
  \left\{\left.C^{(D,\lambda)}_{(L,s)(K,t)}\;\right|\;(D,\lambda)\in\Lambda
    \text{ and }(L,s),(K,t)\in M(D,\lambda)\right\}
\]
form an $R$-basis for $R^\alpha[S]$.
\end{proof}
This verifies property (C1) in Definition \ref{gus}. We next prove property (C2). We already know that $*$ is an $R$-linear anti-involution of 
$R^\alpha[S]$, so we need only check the following.
\begin{lem}
The action of $*$ on the basis elements $C^{(D,\lambda)}_{(L,s)(K,t)}$ is given by
\[
  \left(C^{(D,\lambda)}_{(L,s)(K,t)}\right)^*=C^{(D,\lambda)}_{(K,t)(L,s)}.
\]
\end{lem}
\begin{proof}
By Assumption \ref{assumgrp}, we have $\left(C^\lambda_{st}\right)^*=C^\lambda_{ts}$. Thus
\begin{eqnarray*}
  \left(C^{(D,\lambda)}_{(L,s)(K,t)}\right)^*
    &=&\left(\left(u_L^*\circ C^\lambda_{st}\right)\circ u_K\right)^*\\
    &=&u_K^*\circ\left(\left(C^\lambda_{st}\right)^*\circ u_L\right)\text{ using (\ref{beta3'}) twice}\\
    &=&\left(u_K^*\circ C^\lambda_{ts}\right)\circ u_L\text{ by (\ref{beta1'})}\\
    &=&C^{(D,\lambda)}_{(K,t)(L,s)},
\end{eqnarray*}
as required.
\end{proof}
Suppose $D,D'\in\D$ satisfy $D'<_\DD D$. Pick any $\lambda\in\Lambda_D$. By (\ref{Dbasis}), we have
\begin{eqnarray*}
  R^\alpha[D']
    &=&\text{span}_R\bigg\{C^{(D',\lambda')}_{(L',s')(K',t')}\bigg\}.\\
    &\subseteq&\text{span}_R\bigg\{C^{(D'',\lambda'')}_{(L'',s'')(K'',t'')}\;\bigg|\;D''<_\DD D\bigg\}.\\
    &\subseteq&R^\alpha[S](<(D,\lambda)),
\end{eqnarray*}
where $R^\alpha[S](<(D,\lambda))$ is as defined in Definition \ref{gus}. Thus
\begin{equation}\label{D'}
  \bigoplus_{D'<_\DD D}R^\alpha[D']\subseteq R^\alpha[S](<(D,\lambda)).
\end{equation}
We now prove (C3).
\begin{lem}\label{C3lem}
Given $(D,\lambda)\in\Lambda$ and $(L,s)\in M(D,\lambda)$, and for an element $a\in R^\alpha[S]$, there exist elements $r_a((L',s'),(L,s))\in R$ for 
$(L',s')\in M(D,\lambda)$ such that
\[
  a\cdot C^{(D,\lambda)}_{(L,s)(K,t)}
    \in\sum_{(L',s')\in M(D,\lambda)}r_a((L',s'),(L,s))C^{(D,\lambda)}_{(L',s')(K,t)}+R^\alpha[S](<(D,\lambda))
\]
for each $(K,t)\in M(D,\lambda)$.
\end{lem}
\begin{proof}
Because $S$ spans $R^\alpha[S]$ as an $R$-module, it suffices to take $a\in S$. Because $u_L^*\in D$, clearly $au_L^*\leq_\DD D$. First suppose that 
$au_L^*<_\DD D$. Then $au_L^*gu_K<_\DD D$ for all $g\in G_D$ and $K\in\LL_D$, so (\ref{D'}) gives
\[
  \alpha(a,u_L^*gu_K)c^\lambda_{st}(g)\beta(u_L^*,g)\beta(u_L^*g,u_K)(au_L^*gu_K)\in R^\alpha[S](<(D,\lambda))
\]
for $t\in M_D(\lambda)$. Summing over $g\in G_D$ gives $a\cdot C^{(D,\lambda)}_{(L,s)(K,t)}\in R^\alpha[S](<(D,\lambda))$. It therefore suffices to take 
$r_a((L',s'),(L,s))=0$ for all $(L',s')\in M(D,\lambda)$ in this case.

The other case is when $au_L^*\in D$. It follows from (iii) of Theorem \ref{known} that $au_L^*\;\LL\;u_L^*$, so that $au_L^*\in L_D$. Thus if $L_1^*\in\Rcl_D$ is 
the $\RR$ class of $au_L^*$, then $au_L^*\;\HH\;u_{L_1}^*$. As in the proof of Lemma \ref{C1lem} above, it follows from Green's Lemma that $au_L^*=u_{L_1}^*h$ for 
some $h\in G_D$. By Assumption \ref{assumgrp}, there exist ring elements $r_h(s',s)\in R$ for $s'\in M_D(\lambda)$ such that
\begin{eqnarray*}
  h\circ C^\lambda_{st}-\sum_{s'\in M_D(\lambda)}r_h(s',s)C^\lambda_{s't}
    &\in&R^\beta[G_D](<\lambda)\\
    &=&\spn_R\left\{\left.C^\mu_{uv}\right|\mu<\lambda\text{ and }u,v\in M_D(\mu)\right\}.
\end{eqnarray*}
Applying $u_{L_1}^*\circ$ on the left and ${}\circ u_K$ on the right, we obtain
\begin{eqnarray*}
  \left(u_{L_1}^*\circ\left(h\circ C^\lambda_{st}\right)\right)\circ u_K
      -\sum_{s'\in M_D(\lambda)}r_h(s',s)C^{(D,\lambda)}_{(L_1,s')(K,t)}\hspace{-80mm}\\
    &\in&\spn_R\left\{\left.C^{(D,\mu)}_{(L_1,u)(K,v)}\right|\mu<\lambda\text{ and }u,v\in M_D(\mu)\right\}\\
    &\subseteq&R^\alpha[S](<(D,\lambda)).
\end{eqnarray*}
We can also calculate
\begin{eqnarray*}
  u_{L_1}^*\circ\left(h\circ C^\lambda_{st}\right)
    &=&\left(u_{L_1}^*\circ h\right)\circ C^\lambda_{st}\text{ by (\ref{beta1'})}\\
    &=&\beta(u_{L_1}^*,h)\left(u_{L_1}^*h\right)\circ C^\lambda_{st}\\
    &=&\beta(u_{L_1}^*,h)\left(au_L^*\right)\circ C^\lambda_{st}.
\end{eqnarray*}
Combining these, we obtain
\begin{eqnarray*}
  a\cdot C^{(D,\lambda)}_{(L,s)(K,t)}
    &=&a\cdot\left(u_L^*\circ\left(C^\lambda_{st}\circ u_K\right)\right)\\
    &=&\left(a\cdot u_L^*\right)\circ\left(C^\lambda_{st}\circ u_K\right)\text{ by (\ref{beta2'})}\\
    &=&\alpha(a,u_L^*)\left(au_L^*\right)\circ\left(C^\lambda_{st}\circ u_K\right)\\
    &=&\alpha(a,u_L^*)\left(\left(au_L^*\right)\circ C^\lambda_{st}\right)\circ u_K\text{ by (\ref{beta1'})}\\
    &=&\alpha(a,u_L^*)\beta(u_{L_1}^*,h)^{-1}\left(u_{L_1}^*\circ\left(h\circ C^\lambda_{st}\right)\right)\circ u_K\\
    &\in&\alpha(a,u_L^*)\beta(u_{L_1}^*,h)^{-1}\sum_{s'\in M_D(\lambda)}r_h(s',s)C^{(D,\lambda)}_{(L_1,s')(K,t)}\\
    &&\hspace{20mm}+R^\alpha[S](<(D,\lambda)).
\end{eqnarray*}
It therefore suffices to take
\[
  r_a((L',s'),(L,s))=\begin{cases}
    \alpha(a,u_L^*)\beta(u_{L_1}^*,h)^{-1}r_h(s',s)&\text{if }L'=L_1\\
    0&\text{if }L'\neq L_1.
  \end{cases}
\]
\end{proof}
\section{Linear Representations of Regular Semigroups}
Section 2 of \cite{gus} describes how to construct \emph{cell representations} of a cellular algebra $A$ from its cell 
datum $(\Lambda,M,C,*)$, and defines bilinear forms $\phi^\lambda$ associated with these representations. For convenience 
we reproduce the definitions here. For each $\lambda\in\Lambda$, the cell representation $W(\lambda)$ corresponding to 
$\lambda$ is the left $A$-module with $R$-basis $\{C_s\mid s\in M(\lambda)\}$ and $A$-action
\[
  aC_s=\sum_{s'\in M(\lambda)}r_a(s',s)C_{s'}
\]
for $a\in A$ and $s\in M(\lambda)$. We use
\[
  \rho^\lambda:A\rightarrow\mat_{M(\lambda)}(R)
\]
to denote the corresponding representation relative to the natural basis. That is,
\[
  \rho^\lambda(a)_{st}=r_a(s,t)
\]
for $a\in A$ and $s,t\in M(\lambda)$. For each $a\in A$, the bilinear form $\phi_a^\lambda$ on $W(\lambda)$ 
is defined on the basis elements so that $\phi_a^\lambda(C_s,C_t)$ is the unique element of $R$ satisfying
\begin{equation}\label{tst's'}
  C^\lambda_{s's}aC^\lambda_{tt'}\in\phi_a^\lambda(C_s,C_t)C^\lambda_{s't'}+A(<\lambda)
\end{equation}
for all $s',t'\in M(\lambda)$. This is extended to be $R$-bilinear. We are most interested in the bilinear form
\[
  \phi^\lambda=\phi_1^\lambda.
\]
We use $\Phi^\lambda$ to denote the matrix representation of $\phi^\lambda$ relative to the natural basis. That is, 
$\Phi^\lambda\in\mat_{M(\lambda)}(R)$ is defined by
\[
  \Phi^\lambda_{st}=\phi^\lambda(C_s,C_t)
\]
for $s,t\in M(\lambda)$. In fact $\phi_a^\lambda$ can be related to $\phi^\lambda$ and $\rho^\lambda$ using (C3) of Definition 
\ref{gus}. More precisely,
\begin{equation}\label{phiarho}
  \phi_a^\lambda(C_s,C_t)
    =\sum_{t'\in M(\lambda)}\Phi^\lambda_{st'}\rho^\lambda(a)_{t't}
\end{equation}
for $a\in A$ and $s,t\in M(\lambda)$. The importance of $\phi^\lambda$ is demonstrated by the following theorem.
\begin{thm}[\cite{gus} Theorem 3.8]\label{gussemis}
In the above notation, if $R$ is a field then the following are equivalent.
\begin{enumerate}
\item The algebra $A$ is semisimple.
\item The nonzero cell representations $W(\lambda)$ are irreducible and pairwise inequivalent.
\item The form $\phi^\lambda$ is nondegenerate (ie $\det\Phi^\lambda\neq0$) for each $\lambda\in\Lambda$.
\end{enumerate}
\end{thm}
Theorem \ref{main} allows us to obtain cell representations of $R^\alpha[S]$ from the cell representations of the twisted 
group algebras $R^\beta[G_D]$. In fact this is a special case of the following general result, the proof of which is a 
consequence of Theorem \ref{known} and Green's Lemma, and is omitted (see also Theorem 2.3 of \cite{McAlister}).
\begin{prop}
Suppose that $S$ is any group bound semigroup and that $D$ is a regular $\DD$ class in $S$ with maximal subgroup $G$. Suppose 
$\alpha$ is a twisting from $S$ into $R$, and
\[
  \beta:L_D\times K_D\rightarrow G(R)
\]
is a map satisfying (\ref{beta1}) and (\ref{beta2}), where $L_D$ is the $\LL$ class of $G$ and $K_D$ is the 
$\RR$ class of $G$. Suppose that $M$ is a left $R^\beta[G]$-module. For each $K\in\Rcl_D$, pick an element $v_K\in K$ in the 
same $\LL$ class as $G$, and let
\[
  M_K=\{m_K\mid m\in M\}
\]
be a set in bijection with $M$. Then
\[
  W=\bigoplus_{K\in\Rcl_D}M_K
\]
is a left $R^\alpha[S]$-module under the action which is defined on $S$ by
\[
  s\cdot m_K=\begin{cases}
    0&\text{if }sv_K<_\DD D\\
    \alpha(s,v_K)\beta(v_{K'},g)^{-1}(gm)_{K'}&\text{if }sv_K=v_{K'}g\text{ where }g\in G
  \end{cases}
\]
for $m\in M$, $K\in\Rcl_D$ and $s\in S$, and which is extended to $R^\alpha[S]$ by $R$-linearity.
\end{prop}
Now suppose that the assumptions of Theorem \ref{main} hold. By analogy with Section 4 of \cite{james}, we determine 
the bilinear forms associated with the cell representations of $R^\alpha[S]$ in terms of the cell representations of 
$R^\beta[G_D]$. For any $L,K\in\Lcl_D$ it follows from (ii) and (iii) of Theorem \ref{known} that the element $u_L^{}u_K^*$ is 
either in $G_D$ or in a lower $\DD$ class than $D$. We can therefore define the matrix 
$P_D^\alpha\in\mat_{\Lcl_D}(R^\beta[G_D])$ by
\[
  (P_D^\alpha)_{LK}=\begin{cases}
    0&\text{if }u_L^{}u_K^*<_\DD D\\
    \alpha(u_L^{},u_K^*)u_L^{}u_K^*&\text{if }u_L^{}u_K^*\in G_D.
  \end{cases}
\]
Call $P_D^\alpha$ the \emph{twisted sandwich matrix} of $D$. Of course when $\alpha$ is trivial, this reduces to the usual sandwich 
matrix, on identifying $G_D\cup\{0\}$ with a subset of $R[G_D]$. We can now state the analogue of Lemma 16 of \cite{james}.
\begin{lem}\label{phi}
Let $(D,\lambda)\in\Lambda$ and $(L,s),(K,t)\in M(D,\lambda)$. Then
\[
  \phi^{(D,\lambda)}\left(C_{(L,s)},C_{(K,t)}\right)
    =\phi_{(P_D^\alpha)_{LK}}^\lambda\left(C_s,C_t\right).
\]
\end{lem}
\begin{proof}
Suppose first that $u_L^{}u_K^*<_\DD D$, so $(P_D^\alpha)_{LK}=0$. Then (\ref{D'}) gives
\[
  u_K^*gu_L^{}u_K^*hu_L^{}\in R^\alpha[S](<(D,\lambda))
\]
for $g,\,h\in G_D$. Multiplying by
\[
  \alpha(u_K^*gu_L^{},u_K^*hu_L^{})c^\lambda_{ts}(g)\beta(u_K^*,g)\beta(u_K^*g,u_L^{})c^\lambda_{ts}(h)\beta(u_K^*,h)\beta(u_K^*h,u_L^{})
\]
and summing over $g$ and $h$, we obtain
\[
  C^{(D,\lambda)}_{(K,t)(L,s)}\cdot C^{(D,\lambda)}_{(K,t)(L,s)}
    \in R^\alpha[S](<(D,\lambda)).
\]
Thus
\[
  \phi^{(D,\lambda)}\left(C_{(L,s)},C_{(K,t)}\right)=0=\phi_{(P_D^\alpha)_{LK}}^\lambda\left(C_s,C_t\right)
\]
in this case. The other case is when $u_L^{}u_K^*\in G_D$. Then $(P_D^\alpha)_{LK}=u_L^{}\cdot u_K^*$, so (\ref{beta2'}) gives
\[
  (P_D^\alpha)_{LK}\circ C^\lambda_{ts}=\left(u_L^{}\cdot u_K^*\right)\circ C^\lambda_{ts}
    =u_L^{}\cdot\left(u_K^*\circ C^\lambda_{ts}\right).
\]
Now $u_L^{}u_K^*g\in G_D\subseteq L_D^*$ for all $g\in G_D$. Thus applying ${}\circ u_L^{}$ on the right,
\begin{eqnarray*}
  \left((P_D^\alpha)_{LK}\circ C^\lambda_{ts}\right)\circ u_L^{}
    &=&\left(u_L^{}\cdot\left(u_K^*\circ C^\lambda_{ts}\right)\right)\circ u_L^{}\\
    &=&u_L^{}\cdot\left(\left(u_K^*\circ C^\lambda_{ts}\right)\circ u_L^{}\right)\text{ by (\ref{beta2'})}\\
    &=&u_L^{}\cdot C^{(D,\lambda)}_{(K,t)(L,s)}.
\end{eqnarray*}
Because $u_L^{}u_K^*g\in G_D$, it follows that $u_L^{}u_K^*gu_L^{}\in L_D^*$ as in the proof of Lemma \ref{C1lem}. Thus applying
$\left(u_K^*\circ C^\lambda_{ts}\right)\circ{}$ on the left gives
\begin{eqnarray}
  \left(u_K^*\circ C^\lambda_{ts}\right)\circ\left(\left((P_D^\alpha)_{LK}\circ C^\lambda_{ts}\right)\circ u_L^{}\right)\label{LHS}
    \hspace{-40mm}\\
    &=&\left(u_K^*\circ C^\lambda_{ts}\right)\circ\left(u_L^{}\cdot C^{(D,\lambda)}_{(K,t)(L,s)}\right)\nonumber\\
    &=&\left(\left(u_K^*\circ C^\lambda_{ts}\right)\circ u_L^{}\right)\cdot C^{(D,\lambda)}_{(K,t)(L,s)}\text{ by (\ref{beta4'})}\nonumber\\
    &=&C^{(D,\lambda)}_{(K,t)(L,s)}\cdot C^{(D,\lambda)}_{(K,t)(L,s)}\nonumber.
\end{eqnarray}
Now by definition of $\phi^\lambda$, we have
\[
  C^\lambda_{ts}\circ(P_D^\alpha)_{LK}\circ C^\lambda_{ts}
    \in\phi_{(P_D^\alpha)_{LK}}^\lambda(C_s,C_t)C^\lambda_{ts}+R^\beta[G_D](<\lambda).
\]
As in the proof of Lemma \ref{C3lem}, applying $u_K^*\circ{}$ on the left and ${}\circ u_L$ gives
\begin{eqnarray*}
  \left(u_K^*\circ\left(C^\lambda_{ts}\circ(P_D^\alpha)_{LK}\circ C^\lambda_{ts}\right)\right)\circ u_L^{}
    &\in&\phi_{(P_D^\alpha)_{LK}}^\lambda(C_s,C_t)C^{(D,\lambda)}_{(K,t)(L,s)}\\
    &&\hspace{10mm}+R^\alpha[S](<(D,\lambda)).
\end{eqnarray*}
By applying (\ref{beta1'}) repeatedly, the left hand side is exactly (\ref{LHS}). Therefore
\[
  C^{(D,\lambda)}_{(K,t)(L,s)}\cdot C^{(D,\lambda)}_{(K,t)(L,s)}
    \in\phi_{(P_D^\alpha)_{LK}}^\lambda(C_s,C_t)C^{(D,\lambda)}_{(K,t)(L,s)}+R^\alpha[S](<(D,\lambda)),
\]
whence the result.
\end{proof}
For each $\lambda\in\Lambda_D$, the representation
\[
  \rho^\lambda:R^\beta[G_D]\rightarrow\mat_{M_D(\lambda)}(R)
\]
naturally induces a homomorphism
\[
  \mat_{\Lcl_D}(R^\beta[G_D])\rightarrow\mat_{\Lcl_D}\left(\mat_{M_D(\lambda)}(R)\right)
    \cong\mat_{M(D,\lambda)}(R),
\]
which we also denote by $\rho^\lambda$. 
\begin{cor}\label{Phi}
The matrix representation of $\phi^{(D,\lambda)}$ is given by
\[
  \Phi^{(D,\lambda)}=\Phi'^\lambda\rho^\lambda(P^\alpha_D),
\]
where $\Phi'^\lambda$ is the block diagonal matrix
\[
  \Phi'^\lambda=\begin{pmatrix}
    \Phi^\lambda&0&0&\cdots&0\\
    0&\Phi^\lambda&0&\cdots&0\\
    0&0&\Phi^\lambda&\cdots&0\\
    \vdots&\vdots&\vdots&\ddots&\vdots\\
    0&0&0&\cdots&\Phi^\lambda
  \end{pmatrix}
  \in\mat_{\Lcl_D}\left(\mat_{M_D(\lambda)}(R)\right)
    \cong\mat_{M(D,\lambda)}(R).
\]
Thus
\[
  \det\Phi^{(D,\lambda)}=\left(\det\Phi^\lambda\right)^{|\Lcl_D|}\det\rho^\lambda(P^\alpha_D).
\]
\end{cor}
\begin{proof}
Using (\ref{phiarho}), the previous Lemma gives
\begin{eqnarray*}
  \Phi^{(D,\lambda)}_{(L,s)(K,t)}
    &=&\phi^{(D,\lambda)}\left(C_{(L,s)},C_{(K,t)}\right)\\
    &=&\phi_{(P_D^\alpha)_{LK}}^\lambda(C_s,C_t)\\
    &=&\sum_{t'\in M_D(\lambda)}\Phi^\lambda_{st'}\rho^\lambda\left((P_D^\alpha)_{LK}\right)_{t't}\\
    &=&\sum_{t'\in M_D(\lambda)}\Phi^\lambda_{st'}\rho^\lambda\left(P_D^\alpha\right)_{(L,t')(K,t)}\\
    &=&\sum_{(L',t')\in M(D,\lambda)}\Phi^\lambda_{st'}\delta_{LL'}
      \rho^\lambda\left(P_D^\alpha\right)_{(L',t')(K,t)}\\
    &=&\sum_{(L',t')\in M(D,\lambda)}\Phi'^\lambda_{(L,s)(L',t')}
      \rho^\lambda(P_D^\alpha)_{(L',t')(K,t)}\\
    &=&\left(\Phi'^\lambda\rho^\lambda(P_D^\alpha)\right)_{(L,s)(K,t)}.
\end{eqnarray*}
Hence
\[
  \Phi^{(D,\lambda)}=\Phi'^\lambda\rho^\lambda(P_D^\alpha)
\]
as required. Taking the determinant, it is then clear that
\[
  \det\Phi^{(D,\lambda)}=\det\Phi'^\lambda\det\rho^\lambda(P_D^\alpha)
    =\left(\det\Phi^\lambda\right)^{|\Lcl_D|}\det\rho^\lambda(P^\alpha_D).
\]
This completes the proof of Corollary \ref{Phi}.
\end{proof}
The utility of cellular machinery will be illustrated by providing an alternative proof of a special case (Theorem 
\ref{specialcase} below) of the following difficult theorem.
\begin{thm}
Suppose that $S$ is a finite regular semigroup, and suppose $\alpha$ is a twisting from $S$ into some field $R$ such that 
$\alpha(x,y)\neq0$ for each $x,y\in S$. Consider a $\DD$ class $D$ in $S$, and choose any idempotent $1_D\in D$. The $\HH$ class 
$G_D$ of $1_D$ is a group. For each $L\in\Lcl_D$, pick an element $u_L\in L$ with $u_L\;\RR\;1_D$. Similarly for $K\in\Rcl_D$, 
pick $v_K\in K$ with $v_K\;\LL\;1_D$. The twisted sandwich matrix $P_D^\alpha$ is the $\Lcl_D\times\Rcl_D$ matrix with entries in 
$R^\alpha[G_D]$ given by
\[
  (P_D^\alpha)_{LK}=\begin{cases}
    0&\text{if }u_Lv_K<_\DD D\\
    \alpha(u_L,v_K)u_Lv_K&\text{if }u_Lv_K\in G_D.
  \end{cases}
\]
Then $R^\alpha[S]$ is semisimple exactly when the following two conditions hold for each $\DD$ class $D$.
\begin{enumerate}
\item $R^\alpha[G_D]$ is semisimple.
\item $P_D^\alpha$ is square and invertible.
\end{enumerate}
\end{thm}
This result is exactly analogous to the well known non-twisted version \cite{Munn}. Indeed it is easy to check that if 
$S_1\subseteq S_2$ are ideals of $S$ such that $S_2\setminus S_1$ is a single $\DD$ class $D$, then the quotient
\[
  R^\alpha[S_2]/R^\alpha[S_1]\cong R^\alpha_0[S_2/S_1]
\]
is a Munn ring over the ring $R^\alpha[G_D]$, with sandwich matrix $P_D^\alpha$; here the notation $R^\alpha_0[S_2/S_1]$ 
denotes the contracted twisted semigroup algebra, defined analogously to a contracted semigroup algebra. The above 
theorem then follows from Theorem 4.7 of \cite{Munn} (see also \cite{Pon}).

If the assumptions of Corollary \ref{mainunits} hold, the resulting cellular structure is sufficient by itself to quickly obtain 
the above theorem from general cellular algebra results, as we see below. Note that setting $v_{L^*}=u_L^*$, the definition 
of $P_D^\alpha$ given before Lemma \ref{phi} agrees with that in the above theorem.
\begin{thm}\label{specialcase}
Suppose that the conditions of Corollary \ref{mainunits} hold, and that $R$ is a field. Then $R^\alpha[S]$ is semisimple exactly 
when
\begin{enumerate}
\item $R^\alpha[G_D]$ is semisimple and
\item $P_D^\alpha$ is invertible,
\end{enumerate}
for each $D\in\D$, where $P_D^\alpha$ is as defined immediately before Lemma \ref{phi}.
\end{thm}
\begin{proof}
Suppose that the two conditions hold, and consider any $(D,\lambda)\in\Lambda$. Since $P_D^\alpha$ is invertible, 
certainly $\rho^\lambda(P_D^\alpha)$ is invertible. Thus $\det\rho^\lambda(P_D^\alpha)\neq0$. Also because 
$R^\alpha[G_D]$ is semisimple, by Theorem \ref{gussemis} we have $\det\Phi^\lambda\neq0$. Hence Corollary \ref{Phi} gives
\[
  \det\Phi^{(D,\lambda)}\neq0.
\]
As this holds for each $(D,\lambda)\in\Lambda$, the algebra $R^\alpha[S]$ is semisimple by Theorem \ref{gussemis}.

Conversely suppose that $R^\alpha[S]$ is semisimple, so that $\det\Phi^{(D,\lambda)}\neq0$ for each $(D,\lambda)\in\Lambda$ 
by Theorem \ref{gussemis}. By Corollary \ref{Phi}, we then have
\[
  \det\Phi^\lambda\neq0\hspace{20mm}\text{and}\hspace{20mm}\det\rho^\lambda(P_D^\alpha)\neq0.
\]
Now the former holds for all $\lambda\in\Lambda_D$. Thus applying Theorem \ref{gussemis}, statement (i) implies that 
$R^\alpha[G_D]$ is semisimple, and moreover statement (iii) implies that the map
\[
  \bigoplus_{\lambda\in\Lambda_D}\rho^\lambda:R^\alpha[G_D]
    \rightarrow\bigoplus_{\lambda\in\Lambda_D}\mat_{M(\lambda)}(R)
\]
is an isomorphism. Because $\det\rho^\lambda(P_D^\alpha)\neq0$, the matrix $\rho^\lambda(P_D^\alpha)$ is invertible for 
each $\lambda\in\Lambda_D$. Thus
\[
  \bigoplus_{\lambda\in\Lambda_D}\rho^\lambda(P_D^\alpha)\in\bigoplus_{\lambda\in\Lambda_D}\mat_{M(D,\lambda)}(R)
\]
is invertible. The above isomorphism then implies that the matrix $P_D^\alpha$ is invertible. Thus both conditions hold, 
verifying the reverse direction and completing the proof of Theorem \ref{specialcase}.
\end{proof}
\section{The Partition Algebra}
Fix an integer $n\geq1$. For convenience, we denote
\begin{eqnarray*}
  I&=&\{1,2,3,\ldots,n\},\\
  I'&=&\{1',2',3',\ldots,n'\},\\
  I''&=&\{1'',2'',3'',\ldots,n''\}.
\end{eqnarray*}
Let $A_n$ denote the set of equivalence relations on the set $I\cup I'$. For $x\in A_n$, let $\tilde x$ denote the set of 
equivalence classes of $x$. We define a binary operation on $A_n$ as follows. Consider two elements $x,y\in A_n$. Let $y'$ 
denote the equivalence relation on the set $I'\cup I''$ which is obtained from $y$ by 
appending a $'$ to each number. Let $\langle x\cup y'\rangle$ denote the equivalence relation on the set $I\cup I'\cup I''$ 
which is generated by $x$ and $y'$. Let $m(x,y)$ denote the number of equivalence classes of $\langle x\cup y'\rangle$ which 
contain only single dashed elements, that is which are contained in $I'$. Remove all the single dashed elements from 
$\langle x\cup y'\rangle$ and replace the double dashes with single dashes to obtain $xy$. That is, $xy$ is obtained from
\[
  \{(i,j)\in\langle x\cup y'\rangle\mid i,j\in I\cup I''\}
\]
by replacing $i''$ with $i'$. For example, consider the elements $x,y\in A_7$ whose equivalence classes are
\begin{eqnarray*}
  \tilde x&=&\{\{1,3,4',6'\},\{2\},\{4,5,6\},\{7\},\{1'\},\{2',3'\},\{5',7'\}\},\\
  \tilde y&=&\{\{1\},\{2,4\},\{3,3',4',6'\},\{5,7\},\{6,5',7'\},\{1'\},\{2'\}\}.
\end{eqnarray*}
Then
\begin{eqnarray*}
  \widetilde{y'}&=&\{\{1'\},\{2',4'\},\{3',3'',4'',6''\},\{5',7'\},\{6',5'',7''\},\{1''\},\{2''\}\},\\
  \widetilde{\langle x\cup y'\rangle}
    &=&\{\{1,3,2',3',4',6',3'',4'',5'',6'',7''\},\{2\},\{4,5,6\},\\
    &&\hspace{10mm}\{7\},\{1'\},\{5',7'\},\{1''\},\{2''\}\},\\
  \widetilde{xy}&=&\{\{1,3,3',4',5',6',7'\},\{2\},\{4,5,6\},\{7\},\{1'\},\{2'\}\}.
\end{eqnarray*}
Also $m(x,y)=2$ since $\langle x\cup y'\rangle$ has two equivalence classes contained in $I'$, namely $\{1'\}$ and $\{5',7'\}$. This 
operation has a natural diagrammatic interpretation described in \cite{Martin}. It is associative, and we have the relation
\[
  m(x,y)+m(xy,z)=m(x,yz)+m(y,z)
\]
for any $x,y,z\in A_n$. The latter implies that for any $\delta$ in a commutative ring $R$, we can define a twisting from 
$A_n$ into $R$ by
\[
  \alpha(x,y)=\delta^{m(x,y)}.
\]
The resulting twisted semigroup algebra $R^\alpha[A_n]$ is called the \emph{partition algebra} \cite{Martin}. This algebra was 
shown to be cellular by Xi in \cite{xi}. We reproduce this result here with the aid of Theorem \ref{main}.

We first note that $A_n$ has a natural anti-involution $*$ which swaps $i$ and $i'$, for each $i\in I$. It is easy to see that $\alpha$ 
and $*$ satisfy Assumption \ref{twiststar}. Green's relations in $A_n$ are described by the following theorem, the proof of which is 
straightforward and omitted. 
\begin{thm}\label{greenpartition}
For $x\in A_n$, define the functions
\begin{eqnarray*}
  d(x)&=&\#\{J\in\tilde x\mid J\cap I\neq\emptyset\neq J\cap I'\},\\
  r(x)&=&\left(\{J\in\tilde x\mid J\subseteq I\},
    \{J\cap I\mid J\in\tilde x\text{ and }J\cap I\neq\emptyset\neq J\cap I'\}\right),\\
  l(x)&=&\left(\{J\in\tilde x\mid J\subseteq I'\},
    \{J\cap I'\mid J\in\tilde x\text{ and }J\cap I\neq\emptyset\neq J\cap I'\}\right).
\end{eqnarray*}
Then for each $x,y\in A_n$,
\begin{enumerate}
\item $x\;\DD\;y$ exactly when $d(x)=d(y)$.
\item $x\;\RR\;y$ exactly when $r(x)=r(y)$.
\item $x\;\LL\;y$ exactly when $l(x)=l(y)$.
\end{enumerate}
\end{thm}
We note that $r(x)$ and $l(x)$ correspond to elements of the set $S_n(k)$ of \cite{Martin}, where $k=d(x)$. Now $m(x,y)$ 
depends only on the first components of $l(x)$ and $r(y)$. If $y\;\RR\;z$ then $r(y)=r(z)$ by Theorem \ref{greenpartition}, so 
that $\alpha(x,y)=\alpha(x,z)$. Consider a $\DD$ class $D$ in $A_n$. Theorem \ref{greenpartition} implies that $D=d^{-1}(n-k)$ 
for some integer $k$ with $0\leq k\leq n$. Let $1_D$ denote the element of $A_n$ whose equivalence classes are
\[
  \tilde 1_D=\{\{1,2,\ldots,k\},\{1',2',\ldots,k'\}\}\cup\{\{i,i'\}\mid k<i\leq n\}.
\]
It is clear that $1_D\in D$ is an idempotent invariant under $*$. Moreover $x\in G_D$ exactly when
\[
  r(x)=r(1_D)=\left(\{\{1,2,\ldots,k\}\},\{\{i\}\mid k<i\leq n\}\right)
\]
and
\[
  l(x)=l(1_D)=\left(\{\{1',2',\ldots,k'\}\},\{\{i'\}\mid k<i\leq n\}\right).
\]
Thus $x$ differs from $1_D$ only by how the elements $k+1,k+2,\ldots,n$ are paired with the elements $(k+1)',(k+2)',\ldots,n'$. It then 
follows quickly from the multiplication in $A_n$ that there is a group isomorphism $\theta_D$ from the symmetric group $S_{n-k}$ to $G_D$ 
such that
\[
  \widetilde{\theta_D(\sigma)}=\{\{1,2,\ldots,k\},\{1',2',\ldots,k'\}\}
    \cup\{\{k+\sigma(i),(k+i)'\}\mid 1\leq i\leq n-k\}.
\]
Moreover $*$ corresponds under $\theta_D$ to inversion in $S_{n-k}$. From example (1.2) of \cite{gus}, we know that $R[S_{n-k}]$ is cellular 
with the anti-involution induced by inversion. Therefore $R[G_D]$ is cellular with anti-involution $*$. The assumptions of Corollary \ref{mainconst} 
are then satisfied, so the partition algebra $R^\alpha[A_n]$ is cellular.
\section{The Brauer and Temperley-Lieb Algebras}
Suppose Corollary \ref{mainconst} applies to $R^\alpha[S]$, and we wish to apply it to $R^\alpha[T]$, where $T$ is a subsemigroup of $S$ fixed 
setwise by the involution $*$. Restricting $\alpha$ and $*$ to $T$, Assumptions \ref{assum1} and \ref{twiststar} clearly still hold. Moreover if 
$y$ and $z$ are $\RR$ related in $T$, they are certainly $\RR$ related in $S$, so $\alpha(x,y)=\alpha(x,z)$ for $x\in T$. It therefore suffices to 
check Assumption \ref{1D} and that the relevant group algebras are cellular with anti-involution $*$.

Let $BR_n$ denote the set of elements of $A_n$ whose equivalence classes each contain $2$ elements. Thus $BR_n$ essentially 
consists of all partitions of the set $I\cup I'$ into pairs. We represent elements of $BR_n$ as diagrams by arranging $2n$ dots in the plane 
and labelling them as shown below, and joining the pairs with arcs.
\[
  \includegraphics[height=20mm]{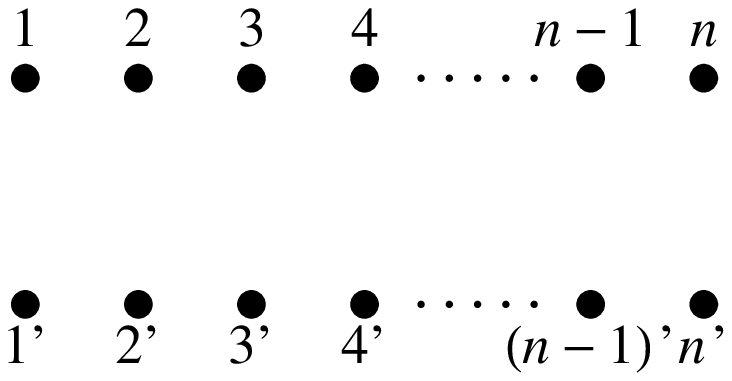}
\]
For example, the element
\[
  x=\{\{1,3\},\{2,5'\},\{4,1'\},\{5,3'\},\{2',4'\}\}
\]
of $BR_5$ is represented by the diagram
\[
  x=\begin{array}[c]{c}\includegraphics[height=30mm]{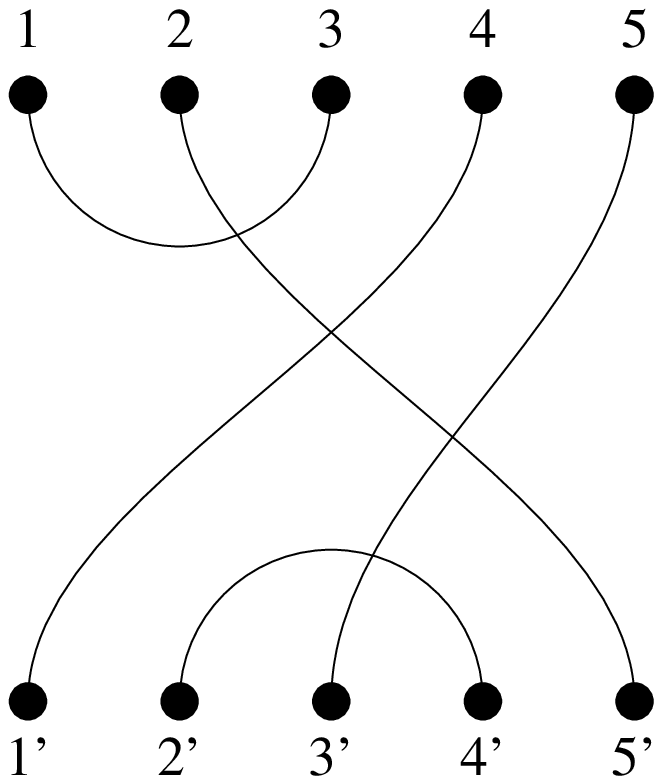}\end{array}.
\]
In fact $BR_n$ forms a subsemigroup of $A_n$, called the \emph{Brauer semigroup} \cite{brauersemi,brauer4}. The twisted semigroup algebra $R^\alpha[BR_n]$ 
is called the \emph{Brauer algebra}. This algebra has been studied extensively in the literature; for example, see \cite{brauer1,brauer2,brauer3}. It was 
realized as a twisted semigroup algebra as above in \cite{brauer4}. The Green's relations in $BR_n$ are described by the following result, given in 
Theorem 7 of \cite{brauersemi}.
\begin{thm}\label{green}
For $x\in BR_n$, define the functions
\begin{eqnarray*}
  r(x)&=&\{\{i,j\}\in x\mid1\leq i,j\leq n\},\\
  l(x)&=&\{\{i',j'\}\in x\mid 1\leq i,j\leq n\},\\
  d(x)&=&\#\{\{i,j'\}\in x\mid1\leq i,j\leq n\}.
\end{eqnarray*}
Note that
\[
  d(x)=n-2|r(x)|=n-2|l(x)|\in\{n,n-2,n-4,\ldots\}.
\]
Suppose $x,y\in BR_n$. Then
\begin{enumerate}
\item $x\;\DD\;y$ exactly when $d(x)=d(y)$.
\item $x\;\RR\;y$ exactly when $r(x)=r(y)$.
\item $x\;\LL\;y$ exactly when $l(x)=l(y)$.
\end{enumerate}
\end{thm}
Now the $\DD$ class $D=d^{-1}(n-2k)$ contains the following idempotent.
\begin{eqnarray*}
  1_D
    &=&\{\{2i-1,2i\}\mid1\leq i\leq k\}\cup\{\{(2i-1)',(2i)'\}\mid 1\leq i\leq k\}\\
    &&{}\cup\{\{i,i'\}\mid2k+1\leq i\leq n\}\\ [10pt]
    &=&\begin{array}[c]{c}\includegraphics[height=20mm]{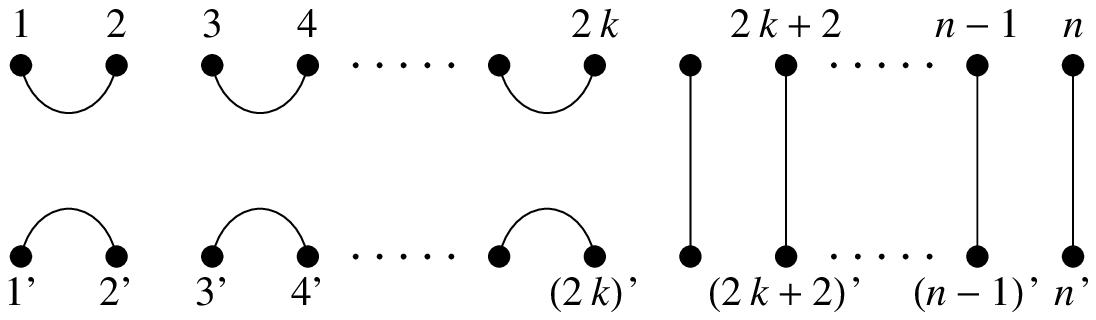}\end{array}.
\end{eqnarray*}
As in the previous section, $1_D$ is fixed by $*$ and its $\HH$ class is isomorphic to the symmetric group $S_{n-2k}$, with $*$ corresponding to 
inversion. By the above discussion, it follows that Corollary \ref{mainconst} applies to the Brauer algebra.

To determine the resulting cell datum, we must choose appropriate elements $u_L$ for each $L\in\Lcl_D$. By Theorem 
\ref{green}, each $L$ is determined uniquely by $l(L)$. Now $l(L)$ consists of $k$ disjoint pairs of elements of the 
set $\{1',2',\ldots,n'\}$. Suppose that the remaining $n-2k$ elements are $\{j_1',j_2',\ldots,j_{n-2k}'\}$, where
\[
  j_1<j_2<\ldots<j_{n-2k}.
\]
Let
\[
  u_L
    =\{\{2i-1,2i\}\mid1\leq i\leq k\}\cup l(L)
      \cup\{\{2k+i,j_i'\}\mid1\leq i\leq n-2k\}.
\]
Diagrammatically, $l(L)$ determines the $k$ edges which have both vertices on the bottom row, while $u_L\;\RR\;1_D$ implies 
that $u_L$ must contain the $k$ edges
\[
  r(1_D)=\{\{2i-1,2i\}\mid1\leq i\leq k\}
\]
which have both vertices on the top row. The last $n-2k$ dots on the top row are joined to the remaining $n-2k$ dots on 
the bottom row in the natural way. For example, suppose that $n=6$ and $k=2$, and consider the $\LL$ class $L$ such that 
$l(L)$ is represented by
\[
  \includegraphics[height=20mm]{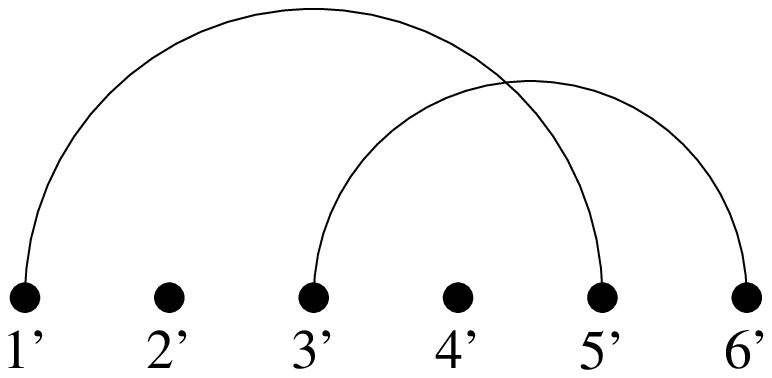}
\]
Then
\[
  u_L=\begin{array}[c]{c}\includegraphics[height=30mm]{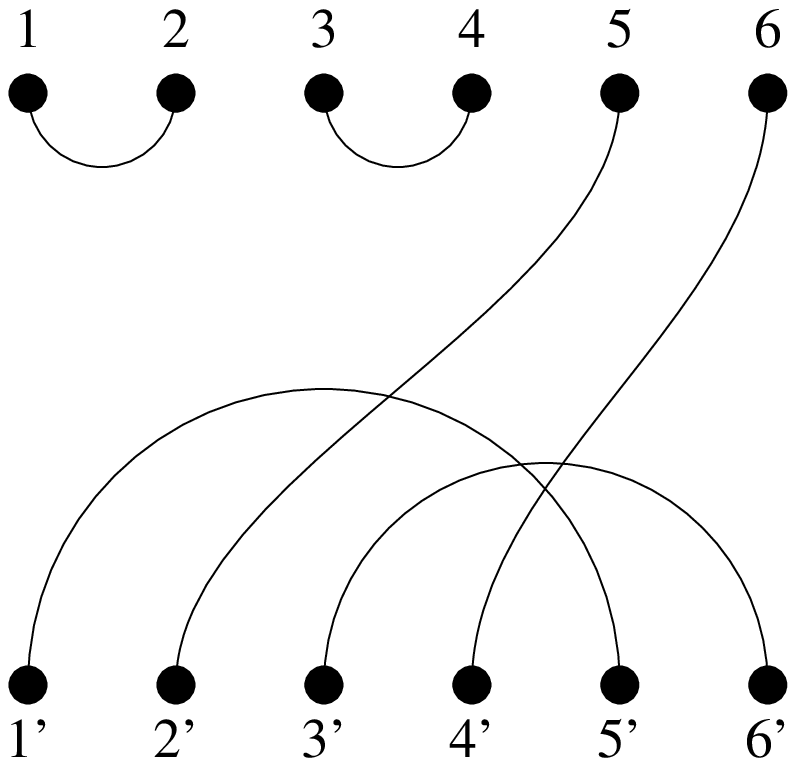}\end{array}.
\]
Having thus defined $u_L$, the cell datum produced by Corollary \ref{mainconst} is exactly that given in \cite{gus}.

The \emph{Temperley-Lieb semigroup} $TL_n$ is the subsemigroup of $BR_n$ consisting of the diagrams that can be 
drawn without intersecting curves. For example, an element of $TL_8$ is shown below.
\[
  \includegraphics[height=35mm]{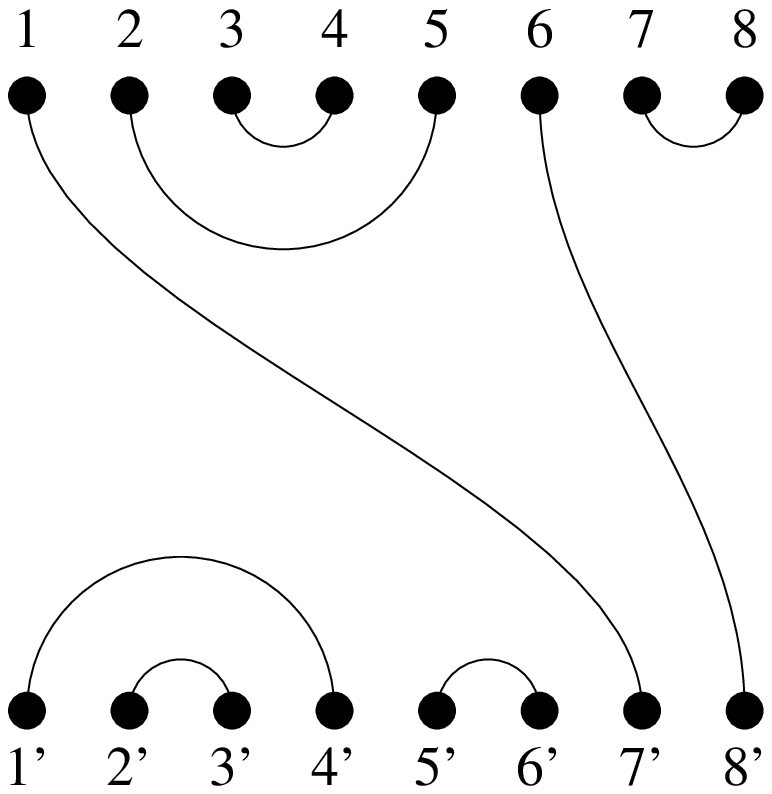}
\]
The twisted semigroup algebra $R^\alpha[TL_n]$ is called the \emph{Temperley-Lieb algebra} \cite{TL1,TL2}. Corollary 
\ref{mainconst} applies to this algebra in the same way. Indeed the $\DD$ classes of $TL_n$ correspond to those of $BR_n$, 
and the idempotents in $BR_n$ constructed above are contained in $TL_n$. The maximal groups are trivial in this case, so 
the group algebras are trivially cellular. Moreover choosing $u_L$ as above, the cell datum produced by Corollary 
\ref{mainconst} is again the same as in \cite{gus}.

The cyclotomic Brauer \cite{cycbrauer} and Temperley-Lieb \cite{cycTL} algebras are variations on the Brauer and Temperley-Lieb 
algebras which depend on an additional positive integer parameter $m$. They were shown to be cellular in \cite{cycbrauercell} and 
\cite{cycTL} respectively, provided the polynomial $x^m-1$ can be decomposed into linear factors over the ground ring $R$. Again 
we can reproduce these results using Corollary \ref{mainconst}. Indeed when realising these algebras as twisted semigroup 
algebras, the underlying semigroups of diagrams have $\DD$ classes corresponding to those in $BR_n$, and idempotents can be chosen 
analogous to those above. In the case of the cyclotomic Brauer algebra, the maximal subgroups are wreath products $\Z_m\wr S_k$, 
the group algebra of which is cellular (with the appropriate anti-involution) by Theorem (5.5) of \cite{gus}. In 
the case of the cyclotomic Temperley-Lieb algebra, the maximal subgroups are direct sums of copies of $\Z_m$, the group algebra of 
which is easily shown to be cellular.
\section{Acknowledgements}
Thanks to James East for helpful discussions and suggestions. Thanks also to my supervisor, David Easdown.
\bibliographystyle{plain}

\end{document}